\documentclass[twoside]{amsart}
\usepackage{latexsym}
\usepackage{amssymb,amsmath,amsopn}
\usepackage[dvips]{graphicx}   
\usepackage{color,epsfig}      
\usepackage{tikz} 
\usepackage{subcaption} 
\usepackage{eurosym}
\newtheorem{thm}{Theorem}

\newtheorem{lem}{Lemma}

\theoremstyle{definition}
\newtheorem{defn}{Definition}
\newtheorem{rem}{Remark}

\renewcommand{\Re}{\mathbb R}

\newcommand{\F}{\mathcal{F}}

\DeclareMathOperator{\conv}{conv}

\DeclareMathOperator{\perim}{perim}
\DeclareMathOperator{\area}{area}
\DeclareMathOperator{\diam}{diam}

\begin{document}

\parskip=4pt

\title[Inscribed polygons]{Extremal convex polygons inscribed in a given convex polygon}

\author[C.L. K\"odm\"on and Z. L\'angi]{Csenge Lili K\"odm\"on and Zsolt L\'angi}
\address{C.L. K\"odm\"on\\ Department of Geometry, Budapest University of Technology, Egry J\'ozsef utca 1., Budapest 1111, Hungary}
\email{kcsl@math.bme.hu}
\address{Z. L\'angi\\ MTA-BME Morphodynamics Research Group and Department of Geometry, Budapest University of Technology, Egry J\'ozsef utca 1., Budapest 1111, Hungary}
\email{zlangi@math.bme.hu}

\thanks{The second author is supported by the National Research, Development and Innovation Office, NKFI, K-119670, the J\'anos Bolyai Research Scholarship of the Hungarian Academy of Sciences, and the BME IE-VIZ TKP2020 and \'UNKP-20-5 New National Excellence Programs by the Ministry of Innovation and Technology.}

\keywords{convex polygon, perimeter, area, billiard, dual billiard}

\begin{abstract}
A convex polygon $Q$ is inscribed in a convex polygon $P$ if
every side of $P$ contains at least one vertex of $Q$. We present algorithms
for finding a minimum area and a minimum perimeter convex polygon inscribed in any given convex
$n$-gon in $O(n)$ and $O(n^3)$ time, respectively. We also investigate other variants of this problem.
\end{abstract}

\subjclass[2010]{52A38, 52B60, 68W01}

\maketitle

\section{Introduction}\label{sec:intro}

Motivated by a problem in statistics, in a recent paper \cite{ADLT19} Ausserhofer, Dann, T\'oth and the second named author examined the algorithmic aspects of finding convex polygons of maximal area, circumscribed about a given convex polygon.
The aim of our paper is to continue this investigation.

Our primary goal is to find, among the convex polygons inscribed in a given convex $n$-gon, one with 
minimum area or perimeter. For these problems we give algorithmic solutions requiring $O(n)$ and $O(n^3)$ steps, respectively.
We will see that the first problem is relatively easy to solve, but the second one is not. This problem can be regarded as a variant of the problem of finding the shortest closed billiard trajectories in a given convex polygon. We note that this problem, proposed also for convex bodies in general, is an extensively studied area of research closely related, among other things, to dynamical systems and symplectic geometries. For more information on this subject, the reader is referred to the papers \cite{ABKS16, AFOR18, Bezdek, Ghomi}, the book \cite{Tabachnikov}, or the video recording of the highly interesting talk \cite{AAspeech} of Artstein-Avidan.

Besides the algorithms, following \cite{ADLT19}, to any minimum area or perimeter convex polygon inscribed in a convex $n$-gon, we assign a sequence from $\{ U, N \}^n$ describing its combinatorial properties, and completely characterize the sequences that are assigned to some such polygon (for more details, see the first paragraphs of Subsections~\ref{subsec:area_comb} and \ref{subsec:perim_comb}).
We also collect observations about the properties of inscribed or circumscribed convex polygons minimizing or maximizing, respectively, some other geometric quantity.

Finally, we remark that the algorithmic aspects of similar problems have been studied in many papers. To give specific examples, we mention the problem of finding maximum area or perimeter convex $k$-gons in a convex $n$-gon \cite{Aggarwal} or in a point set \cite{BDDG85}, or that of finding minimum area $k$-gons in a point set under several geometric constraints \cite{EOR92}, or the problem of finding maximum area triangles enclosed, or minimum area triangles enclosing a given convex $n$-gon \cite{Chandran, KL85, RAMB86}.

We start with the main definition of our paper, which can be regarded as the `dual' of Definition 1 of \cite{ADLT19}. Here and throughout the paper by $\area(K)$ and $\perim(K)$ we denote the area and the perimeter of the convex region $K$, respectively.

\begin{defn}\label{defn:inscribed}
Let $C$ be a convex polygon. If $Q$ is a convex polygon such that every side of $C$ contains at least one vertex of $Q$, we say that $Q$ is \emph{inscribed} in $C$.
Furthermore, we set
\begin{equation}\label{eq:areadefn}
a(C) = \inf \{ \area(Q) :  Q \hbox{ is inscribed in } C \},
\end{equation}
and
\begin{equation}\label{eq:perimdefn}
p(C) = \inf \{ \perim(Q) :  Q \hbox{ is inscribed in } C \}.
\end{equation}
\end{defn}

Note that there is an inscribed polygon of arbitrarily small area in $C$ if and only if $C$ is a convex $n$-gon with $n \leq 4$. Thus,
to avoid degenerate configurations, throughout this paper $C$ always denotes a convex $n$-gon with $n \geq 5$, and vertices $p_1, p_2, \ldots, p_n$ in counterclockwise order. We extend the indices to all integers so that they are understood modulo $n$; i.e. $p_i = p_j$ if and only if $i \equiv j \mod n$.
To any convex polygon $Q$ inscribed in $C$, one can assign a cyclic sequence $s^{C}(Q) \in \{ U, N \}^n$, where the $k$th element $s^{C}_k(Q)$ of $s^{C}(Q)$ is $U$ (used) if and only if $p_k$ is a vertex of $Q$. Here, by a cyclic sequence we mean a sequence in which the indices of the elements are understood mod $n$, and for brevity, if it is clear which convex polygon $C$ denotes, we write $s(Q) = s^{C}(Q)$ and $s^{C}_k(Q)= s_k(Q)$. We call the cyclic sequence defined in this way the \emph{cyclic sequence associated to $Q$}.


The structure of the paper is as follows. In Section~\ref{sec:area} we find the minimum area polygons inscribed in $C$. In Section~\ref{sec:perim} we
consider minimum perimeter polygons inscribed in $C$. Finally, in Section~\ref{sec:others} we collect our results about circumscribed polygons which maximize some geometric quantity.
In our investigation, for any points $x,y \in \Re^2$, we denote by $xy$ the closed segment with endpoints $x,y$, and the length of $xy$ by $|xy|$. We regard points as position vectors, and thus, by $y-x$ we mean the vector pointing from $x$ to $y$.
We denote the convex hull of a set $X$ by $\conv (X)$, and for brevity, we call the relative interior points of a segment \emph{interior points}.

\section{Minimum area convex polygons inscribed in $C$}\label{sec:area}


\subsection{An algorithmic solution}\label{subsec:area_algorithm}

First, we describe the geometric background for our algorithm.

\begin{thm}\label{thm:area}
Let $Q$ be a minimum area convex polygon inscribed in $C$, with vertices $q_1,q_2,\ldots, q_k$ in counterclockwise order. Then the following holds.
\begin{itemize}
\item[(i)] $Q$ has no two consecutive vertices that are interior points of some sides of $C$.
\item[(ii)] If $q_j$ is a vertex of $Q$ contained in the interior of $p_ip_{i+1}$, then the vertices of $Q$ adjacent to $q_j$ are $p_{i-1}$ and $p_{i+2}$, and $p_{i-1}p_{i+2}$ is parallel to $p_ip_{i+1}$.
\item[(iii)] There is a minimum area convex polygon $Q_0$ inscribed in $C$, with vertices $q_1',q_2',\ldots, q_k'$ in counterclockwise order, such that
\begin{itemize}
\item $q_j$ is a vertex of $C$ if and only if $q_j=q_j'$, and
\item if $q_j$ is an interior point of $p_ip_{i+1}$, then $q_j' \in \{ p_i,p_{i+1} \}$. 
\end{itemize}
\end{itemize}
\end{thm}

\begin{proof}
First, we prove (i). For contradiction, assume that $q_j$ and $q_{j+1}$ are two consecutive vertices of $Q$, and $q_j$ and $q_{j+1}$ are interior points of $p_ip_{i+1}$ and $p_{i+1}p_{i+2}$, respectively. Let $q_{j-1}$ denote the vertex of $Q$ adjacent to $q_j$ and different from $q_{j+1}$, and similarly, let $q_{j+2}$ denote the vertex of $Q$ adjacent to $q_{j+1}$ and different from $q_j$. If $p_ip_{i+1}$ is not parallel to $q_{j-1}q_{j+1}$, then one can slide $q_j$ on $p_ip_{i+1}$ in a suitable direction to decrease the area of $Q$. Thus, it follows from the minimality of the area of $Q$ that $p_ip_{i+1}$ and $q_{j-1}q_{j+1}$ are parallel. The property that $p_{i+1}p_{i+2}$ and $q_jq_{j+2}$ are parallel is obtained by a similar argument. Now, let $Q'$ be a convex polygon obtained from $Q$ by replacing $q_j$ by any point $q_j'$ of $p_ip_{i+1}$. Then, by our previous observation, $Q'$ is a minimum area convex polygon inscribed in $C$. On the other hand, $q_j'q_{j+2}$ and $p_ip_{i+1}$ are not parallel, which implies that we may slide $q_{j+1}$ on $p_ip_{i+2}$ in a suitable direction to obtain a convex polygon inscribed in $C$ with area smaller than $\area(Q)$, which contradicts our assumption. Thus, in the following we assume that Q has no two consecutive vertices that are interior points of some sides of $C$.

Now we prove (ii). Let $q_{j-1}$ and $q_{j+1}$ denote the vertices of $Q$ adjacent to $q_j$ such that $q_{j-1} \in p_{i-1}p_i$ and $q_{j+1} \in p_{i+1}p_{i+2}$.
By (i), $q_{j-1}$ is not an interior point of $p_{i-1}p_i$, and hence, we have $q_{j-1} \in \{ p_{i-1},p_i \}$. On the other hand, if $q_{j-1}=p_i$, then the convex hull $Q'$ of all vertices of $Q$ but $q_j$ is a convex polygon inscribed in $C$ with $\area(Q') < \area(Q)$. Thus, we have $q_{j-1}=p_{i-1}$. The equality $q_{j+1}=p_{i+2}$ follows by a similar argument. The fact that $p_{i-1}p_{i+2}$ is parallel to $p_ip_{i+1}$ is obtained by repeating the argument in the previous paragraph.
Finally, (iii) is a straightforward consequence of (i) and (ii).
\end{proof}

Using the idea of the proof of Theorem~\ref{thm:area}, it is easy to see that if $p_{i-1},p_i,p_{i+1}$ and $p_{i+2}$ are vertices of $P$ where $p_ip_{i+1}$ is parallel to $p_{i-1}p_{i+2}$, and there is a minimum area polygon $Q$ whose vertex set contains $p_{i-1},p_i$ and $p_{i+2}$, then replacing $p_i$ with any point of $p_ip_{i+1}$ we obtain a minimum area polygon, and the same holds after repeating such a modification arbitrarily many times. Our next remark, which can be regarded as a converse of (iii) of Theorem~\ref{thm:area} and can be proved in a similar way, states that using this procedure and starting with the minimum area polygons whose each vertex is a vertex of $P$, one can generate all minimum area polygons inscribed in $C$.

\begin{rem}\label{rem:area}
Let $Q$ be a minimum area convex polygon inscribed in $C$, with vertices $q_1, q_2, \ldots, q_k$ in counterclockwise order, such that every vertex of $Q$ is a vertex of $C$.
Let $1 \leq s_1 < s_2 < \ldots < s_m \leq k$ such that for every value of $t$, $|s_{t+1}-s_t| \geq 2$, and for any value of $t$ there is some index $i_t$ such that
$q_{s_t-1}=p_{i_t-1}$, $q_{s_t+1}= p_{i_t+2}$ and $q_{s_t} \in \{ p_{i_t},p_{i_t+1} \}$. Assume that for $t=1,2,\ldots,m$, $p_{i_t}p_{i_t+1}$ is parallel to $p_{i_t-1}p_{i_t+2}$, and let $q_{s_t}'$ be an arbitrary point $p_{i_t}p_{i_t+1}$. For any $s \notin \{ s_1, \ldots, s_m\}$, set $q_s'=q_s$. Then the convex hull $Q'$ of the points $q_1',q_2',\ldots, q_k'$ is a minimum area convex polygon inscribed in $C$.
\end{rem}

By Theorem~\ref{thm:area} and Remark~\ref{rem:area}, for any convex polygon $C$ there is a minimum area polygon $Q$ inscribed in $C$ with the additional property that every vertex of $Q$ is a vertex of $C$, and by determining all minimum area polygons with this additional property one can determine all minimum area polygons not satisfying this property. Indeed, Remark~\ref{rem:area} yields a sufficient condition for the vertices of these polygons, while it follows from Theorem~\ref{thm:area} that this condition is also necessary.
Thus, in the following we consider only convex hulls of subsets of the vertex set of $C$.

Let $Q$ be such a minimum area convex polygon. Then every side of $Q$ is either a side or a diagonal of $C$. These diagonals of $C$ must lie between vertices separated by a single vertex, as each side of $Q$ must contain at least one vertex of $C$. Let $T_i$ denote the area of the triangle with vertices $p_{i-1},p_i,p_{i+1}$. Then the problem of minimizing the area of $Q$ is equivalent to the problem of choosing some elements of the set $\{ T_1, \ldots, T_n \}$ whose sum is maximal under the restriction that no two elements with consecutive indices are chosen mod $n$. We denote this maximal value by $A$, and present an algorithm that finds the value of $A$ and a subsequence with sum equal to $A$.

First, we compute the values of all $T_i$s. Note that computing the areas of $n$ triangles, using suitable determinants, can be done in $O(n)$ steps.
Let us divide the possible subsequences $S$ into two types: If $S$ contains $T_1$ we say that it is of Type 1, and otherwise it is of Type 2.
For any $1 \leq k \leq n-1$, we denote by $A^1_k$ the maximum of the sums of the elements of subsequences of $T_1, T_2, \ldots, T_k$ containing no two elements with consecutive indices but containing $T_1$, and for any $2 \leq k \leq n$ we denote by $A^2_k$ the maximum of the sums of the elements of subsequences of $T_2, \ldots, T_k$ containing no two elements with consecutive indices. Then, clearly, $A= \max \{ A^1_{n-1}, A^2_n \}$.
We find the values $A^1_k$ and $A^2_k$ using a recursive algorithm.

Note that $A^1_1 = A^1_2 = T_1$, and $A^1_3 = T_1 + T_3$. In general, for any $4 \leq k \leq n-1$, we have $A^1_{k} = \max \{A^1_{k-1}, A^1_{k-2}+T_{k}\}$. We obtain similarly that $A^2_2=T_2$, $A^2_3 = \max \{ T_2, T_3 \}$, and $A^2_{k} =  \max \{A^2_{k-1}, A^2_{k-2}+T_{k}\}$ for all $4 \leq k \leq n$. Finally, observe that $A^1_{n-1}$, $A^2_n$, and also
\[
a(C) = \area(C) - \max \{ A^1_{n-1}, A^2_n \}
\]
can be computed in $O(n)$ steps.

\begin{rem}\label{rem:area_dectree}
Our algorithm can clearly be carried out in such a way that we keep track of all minimum area inscribed polygons whose vertices are the vertices of $C$.
These polygons are not listed by the algorithm (as their number might be even more than linear), but are represented in the form of a decision tree.
\end{rem}



\subsection{Combinatorial properties}\label{subsec:area_comb}

As in the algorithm in Subsection~\ref{subsec:area_algorithm}, in this subsection we investigate minimum area convex polygons inscribed in $C$ with the additional property that all their vertices are vertices of $C$. Our aim is to characterize the family of cyclic sequences that are associated to some such polygon $Q$ for some suitably chosen $C$, and we note that if every vertex of an inscribed polygon $Q$ is a vertex of $C$, then $s(Q)$ determines $Q$.

Clearly, no sequence $s(Q)$ contains two consecutive $N$s. Indeed, if  $s_{k}(Q)=s_{k+1}(Q)=N$, then $Q$ is disjoint from $p_kp_{k+1}$, which contradicts the condition that $Q$ is inscribed in $C$. Similarly, $s(Q)$ contains no three consecutive $U$s, since if $s_{k-1}(Q)=s_{k}(Q)=s_{k+1}(Q)=U$, then the area of $Q$ could be reduced further by not using $p_k$. Our main result in this subsection is the converse of this observation.

\begin{thm}\label{thm:area_combin}
Let $s \in \{ N, U \}^n$ with $n \geq 5$. Then the following are equivalent.
\begin{itemize}
\item[(i)] There is some convex $n$-gon $C$ with a unique minimum area convex polygon $Q$ such that $s(Q)=s$.
\item[(ii)] The cyclic sequence $s$ contains no two consecutive $N$s and no three consecutive $U$s.
\end{itemize}
\end{thm}

\begin{proof}
We only need to prove that (ii) implies (i). Let $k$ denote the number of $N$s in $s$, and observe that since $n \geq 5$, (ii) implies that $k \geq 2$.

First, we present a construction for the case $k \geq 3$.
Let $Q_0$ be  a regular $k$-gon, let the vertices of $Q_0$ be $q_1, q_2, \ldots, q_k$ in counterclockwise order, and let $g_i$ be the midpoint of $q_iq_{i+1}$ for $i=1,2,\ldots,k$. Set $G = \conv \{ g_1, g_2, \ldots, g_k \}$. Choose some arbitrary small value $\varepsilon > 0$.

By the conditions in (ii), for any $1 \leq i \leq k$, there is either one or two $U$s between the two $N$s in $s$ corresponding to $q_i$ and $q_{i+1}$. If there is one $U$ between them, we glue an isosceles triangle to $Q_0$ with $q_iq_{i+1}$ as its base such that the new vertex $q_i'$ is closer to $g_i$ than $\varepsilon$. Similarly, if there are two $U$s between the two $N$s corresponding to $q_i$ and $q_{i+1}$, we glue a symmetric trapezoid to $Q_0$, with $q_iq_{i+1}$ as its base such that the two new vertices $q_i',q_i''$ are closer to $q_iq_{i+1}$ than $\varepsilon$, and $q_i'q_i''$ is parallel to $q_iq_{i+1}$ and its length is less than $\frac{|q_i'g_i|}{2}= \frac{|q_i''g_i|}{2}$. We carry out this operation for all values of $i$, and obtain an $n$-gon, which we denote by $C$ (cf. Figure~\ref{fig:area_const}). Similarly, we denote the convex hull of the points $q_i'$ and $q_i''$ by $Q$. Note that if $\varepsilon$ is sufficiently small, removing one point from each pair $\{ q_i', q_i'' \}$, the convex hull $Q^*$ of the remaining vertices of $Q$ contains $G$, and the same statement holds if we replace 
some of the vertices of $Q^*$ with the corresponding midpoints $g_i$.

\begin{figure}[ht]
\begin{center}
\includegraphics[width=0.4\textwidth]{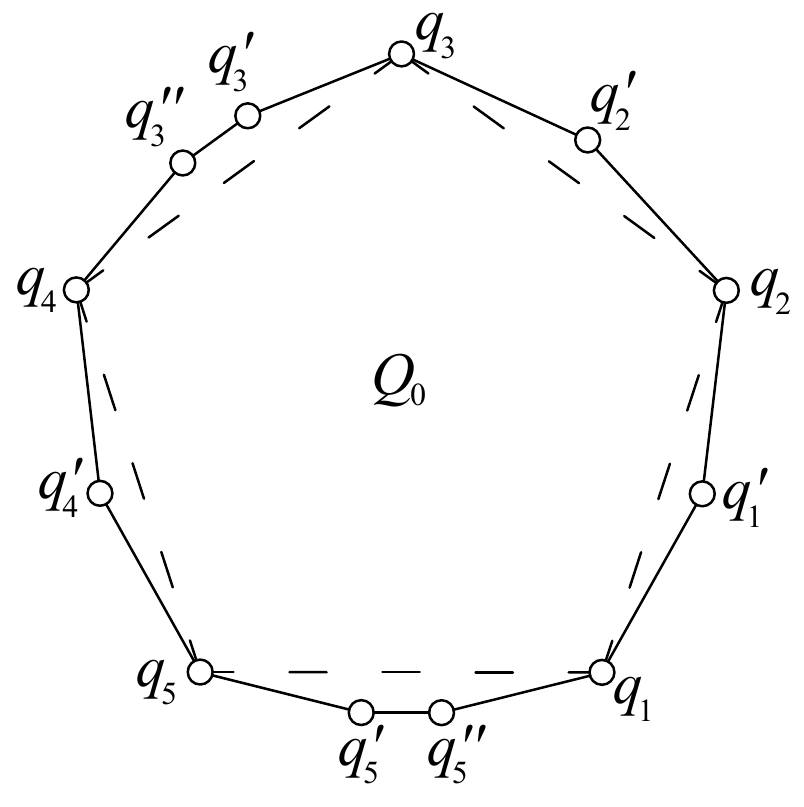}
\caption{The polygon $C$ constructed in the proof of Theorem~\ref{thm:area_combin} for $k=5$ and $s=NUNUNUUNUNUU$.}
\label{fig:area_const}
\end{center}
\end{figure}

We show that $Q$ is the unique minimum area polygon inscribed in $C$ if $\varepsilon$ is sufficiently small. First, observe that in this case $Q$ is convex and its area is close to $\area(G)$. In particular, for sufficiently small values of $\varepsilon$ the inequality $\area(Q) < \area (\conv (G \cup \{ q_i\} ))$ is satisfied for all values of $i$. On the other hand, if $Q'$ is any convex polygon inscribed in $C$ with vertices chosen from the vertices of $C$, then $q_i', q_i''$ or $g_i$ belongs to $Q'$ for all values of $i$. Indeed, if there is one $U$ between the two $N$s corresponding to $q_i$ and $q_{i+1}$, then $q_i'$ is a vertex of $Q'$, or both $q_i$ and $q_{i+1}$ are vertices of $Q'$, implying that $g_i \in Q'$. In the opposite case the fact that $Q'$ is inscribed in $C$ yields that it contains a point of the segment $q_i'q_i''$. Since every vertex of $Q'$ is a vertex of $C$, from this $q_i' \in Q'$ or $q_i'' \in Q'$ follows. But by our previous observation this implies that $G \subseteq Q'$. Thus, if $q_i$ is a vertex of $Q'$ for some value of $i$, then $\area (Q) < \area (\conv (G \cup \{ q_i\} )) \leq \area (Q')$. On the other hand, if no $q_i$ is a vertex of $Q'$, then the facts that every vertex of $Q'$ is a vertex of $C$ and $Q'$ is inscribed in $C$ implies that $Q'=Q$. By Theorem~\ref{thm:area} and Remark~\ref{rem:area}, since $Q$ is a unique minimum area polygon inscribed in $C$ with the additional assumption that every vertex of $Q$ is a vertex of $C$, it follows that there is no minimum area polygon inscribed in $C$ having a vertex in the interior of a side of $C$. Finally, we clearly have $s(Q)=s$, which yields the assertion for $k \geq 3$.
If $k=2$, a similar construction proves the statement, where the regular polygon $Q_0$ is replaced by a straight line segment.
\end{proof}

\section{Minimum perimeter convex polygons inscribed in $C$}\label{sec:perim}

\subsection{An algorithmic solution}\label{subsec:perim_algorithm}

Let $\mathcal{F}(C)$ denote the family of minimum perimeter convex polygons inscribed in $C$.
We start with the description of some properties of the elements of $\mathcal{F}(C)$.
We first prove Lemma~\ref{lem:reflection}, and note that the property described in it is well known in the theory of billiards, and it can be proved also via a simple differential geometric argument.

\begin{lem}\label{lem:reflection}
Let $Q$ be a convex polygon with minimum perimeter inscribed in $C$. Let $q_{j-1},q_j,q_{j+1}$ be three consecutive vertices of $Q$. If $q_j$ is an interior point of a side $p_ip_{i+1}$ of $C$, then $Q$ satisfies the \emph{optic reflection law} at $q_j$, i.e. the angles $\angle p_iq_jq_{j-1}$ and $\angle q_{j+1}q_jp_{i+1}$ are equal.
\end{lem}

\begin{proof}
The locus of the points in the plane with the property that the sum of their distances from $q_{j-1}$ and $q_j$ is a given constant is an ellipse with $q_{j-1}$ and $q_{j+1}$ as its foci. Let $E$ be the ellipse with $q_{j-1}$ and $q_{j+1}$ as its foci, and containing $q_j$ on its boundary. Since $q_j$ minimizes the sum of the distances from $q_{j-1}$ and $q_{j+1}$ among the points of the line $L$ through $p_ip_{i+1}$, $L$ is tangent to $E$ at $q_j$. Thus, the equality $\angle p_iq_jq_{j-1} =\angle q_{j+1}q_jp_{i+1}$ follows from the property of ellipses that the tangent line $L$ at $q_j$ bisects the exterior angles of the triangle $\conv \{q_{j-1},q_j,q_{j+1} \}$ at $q_j$ \cite{Besant}.
\end{proof}

Clearly, by Lemma~\ref{lem:reflection}, every vertex of a minimum perimeter inscribed polygon is either a vertex of $C$ or satisfies the reflection law.
To investigate these polygons, we recall from Section~\ref{sec:intro} the notion of a cyclic sequence $s(Q)$ associated to an inscribed polygon $Q$: the $k$th element $s_k(Q)$ of $s(Q)$ is $U$ if and only if $p_k$ is a vertex of $Q$.

Note that if $s(Q)$ contains exactly $k$ $U$s, then $Q$ has exactly $(n-k)$ vertices. It is also worth noting that convex polygons inscribed in $C$ whose every vertex satisfies the optic reflection law are called \emph{Fagnano orbits}, and that a necessary condition for the existence of a Fagnano orbit in an even-sided polygon can be found in \cite{DR09} as Lemma 2.

In the theory of billiards it is well known that any Fagnano orbit in a convex $n$-gon $C$ with $n$ even can be modified in a natural way to construct infinitely many Fagnano orbits; such a configuration is shown in Figure~\ref{fig:Fagnano}. It is easy to see that all these orbits correspond to minimum perimeter polygons inscribed in $C$. On the other hand, no such construction is known for $n$-gons with $n$ odd, and it is an open problem to characterize the convex $n$-gons with $n$ odd in which Fagnano orbits exist \cite{DR09}. Our next theorem answers a related problem.

\begin{thm}\label{thm:perim_struct}
Let $s \in \{ U , N \}^n$. Then the following holds.
\begin{itemize}
\item[(i)] If $n$ is odd or $s \neq \overbrace{NN \ldots N}^{n}$, then there is at most one minimum perimeter polygon $Q \in \F(C)$ with $s(Q) = s$.
\item[(ii)] If $n$ is even and $s=\overbrace{NN \ldots N}^{n}$, then either there is no $Q \in \F(C)$ with $s(Q) = s$, or there are infinitely many. In the latter case, if $Q_1, Q_2 \in \F(C)$ satisfy $s(Q_1)=s(Q_2)=s$, then all corresponding pairs of sides of $Q_1$ and $Q_2$ are parallel (cf. Figure~\ref{fig:Fagnano}).
\end{itemize}
\end{thm}

\begin{figure}[ht]
\begin{center}
\includegraphics[width=0.35\textwidth]{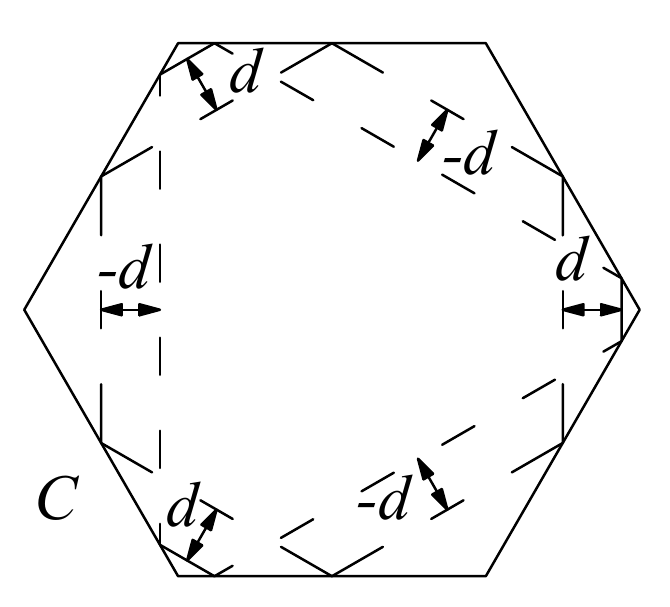}
\caption{Minimum perimeter polygons inscribed in a regular hexagon, indicated with dashed lines. Note that the signed distances of the corresponding sides of these polygons form an alternating sequence.}
\label{fig:Fagnano}
\end{center}
\end{figure}

\begin{proof}
First, we prove (i).

Consider the case that $s$ contains at least two $U$s, and let $p_i$ and $p_j$ be two vertices of $Q$ corresponding to $U$s in $s$ such that $p_{i+1}, p_{i+2},\ldots, p_{j-1}$ correspond to $N$s.
We label the vertices of $Q$ in such a way that $q_i=p_i, q_{i+1} \in p_{i+1}p_{i+2}, \ldots, q_{j-2} \in p_{j-2}p_{j-1}$ and $q_{j-1}=p_j$.
Let $q_i'$ denote the reflected copy of $q_i$ to the line through $p_{i+1}p_{i+2}$. By the optic reflection law, $q_i'$, $q_{i+1}$ and $q_{i+2}$ are collinear, and $|q_i'q_{i+2}| = | q_iq_{i+1}| + |q_{i+1}q_{i+2}|$. Now, if $q_i''$ denotes the reflected copy of $q_i'$ to the line through $p_{i+2}p_{i+3}$, then by the optic reflection law, $q_i''$, $q_{i+2}$ and $q_{i+3}$ are collinear, and $|q_i''q_{i+3}| = |q_iq_{i+1}|+|q_{i+1}q_{i+2}| + |q_{i+2}q_{i+3}|$.
Continuing this process, if $q_i^*$ denotes the point obtained by subsequently reflecting $q_i$ to the lines through $p_{i+1}p_{i+2}, p_{i+2}p_{i+3}, \ldots, p_{j-2}p_{j-1}$, respectively, then $|q_i^*q_{j-1}| = \sum_{t=i}^{j-2} |q_tq_{t+1}|$. Note that $q_i = p_i$ and $q_{j-1}=p_j$, and hence, $q_i^*q_{j-1}$ is depends only on $s(Q)$ and does not depend on $Q$. This yields, in particular, that the length of the boundary of $Q$ from $p_i$ to $p_j$ is independent of $Q$.
On the other hand, for any $i \leq t \leq j-3$, if $q_t^*$ denotes the point obtained by subsequently reflecting $q_t$ to the lines through $p_{t+1}p_{t+2}, \ldots, p_{j-2}p_{j-1}$, then all $q_t^*$s lie on $q_i^*q_{j-1}$. Thus, $q_{j-2}$ is the intersection point of $p_{j-2}p_{j-1}$ and $q_i^*q_{j-1}$, $q_{j-3}$ is the intersection point of $p_{j-3}p_{j-2}$ and the reflected copy of $q_i^*q_{j-2}$ to the line through $p_{j-2}p_{j-1}$, and the remaining $q_t$s can be obtained in a similar way. We note that this process can be carried out in $O(j-i)$ steps, including the check whether the obtained points $q_t$ indeed lie in the interiors of the corresponding sides of $C$.
If $s$ contains exactly one $U$, the same argument can be applied in which the points $p_i$ and $p_j$ coincide, and the part of the boundary of $Q$ between $p_i$ and $p_j$ is equal to the whole boundary of $Q$.

Consider the case that $s =NN \ldots N$, and let $q_0$ be the vertex of $Q$ on $p_np_1$.
For $l \in \{ n, 1 \}$, let $p_l^*$ denote the point obtained by reflecting $p_l$ subsequently about the lines through $p_1p_2, p_2p_3, \ldots, p_{n-1}p_n$, respectively. Let $q(\tau) = \tau p_n + (1-\tau) p_1$, and $q^*(\tau) = \tau p^*_n + (1-\tau) p^*_1$ for any $\tau \in (0,1)$.
By the argument in the previous paragraph, for any value of $\tau$ there is exactly one closed polygonal curve starting at $q(\tau)$, having subsequent vertices $q_1(\tau), \ldots, q_{n-1}(\tau)$ on the lines through $p_1p_2, \ldots, p_{n-1}p_n$, respectively, and returning to $q(\tau)$ such that each vertex $q_i(\tau)$ satisfies the optic reflection law for all $1 \leq i \leq n-1$. This curve, which we denote by $\Gamma(\tau)$, can be obtained by taking the straight line segment $q(\tau)q^*(\tau)$, applying reflections and taking intersections. Thus, the length of the curve $\Gamma(\tau)$ is equal to $|q(\tau)q^*(\tau)|$.

A necessary condition for some $\Gamma(\tau_0)$ to be the boundary of an element of $\F(C)$, which we denote by $Q(\tau_0)$, is that
\begin{itemize}
\item[(a)] all its vertices are contained in the interiors of the sides of $C$;
\item[(b)] it has minimal length among all curves $\Gamma(\tau)$, $\tau \in (0,1)$, satisfying (a).
\end{itemize}
Let $I \subseteq (0,1)$ denote the values of $\tau$ such that $\Gamma(\tau)$ satisfies (a). We need to find the minimum of the length of $\Gamma(\tau)$ on $I$.
Since for all values of $\tau$, the length of $\Gamma(\tau)$ is $|q(\tau)q^*(\tau)|$, it is sufficient to do it for $|q(\tau)q^*(\tau)|$. Note that $|q(\tau)q^*(\tau)|^2 = a \tau^2 + b\tau + c$ for some $a \geq 0, b,c, \in \Re$. Furthermore, as $I$ is clearly open in $[0,1]$, any minimum of $|q(\tau)q^*(\tau)|^2$ on $I$ is a local minimum of $|q(\tau)q^*(\tau)|^2$ on $(0,1)$. Thus, we have that either there is at most one $\Gamma(\tau)$ satisfying (a) and (b), or $|q(\tau)q^*(\tau)|$ is independent of $\tau$. Assume the latter. Then an elementary consideration shows that the vector $q^*(\tau)-q(\tau)$ is independent of $\tau$, which yields that the corresponding pairs of sides of $Q(\tau_1)$ and $Q(\tau_2)$ are parallel for all $\tau_1, \tau_2 \in I$, where $Q(\tau)$ denotes the convex polygon bounded by the closed polygonal curve $\Gamma(\tau)$ for all $\tau \in I$.

Let $Q(\tau_0) \in \F(C)$ for some $\tau_0 \in I$. To complete the proof we show that there is some neighborhood $W$ of $\tau_0$ in $(0,1)$ such that if $n$ is even then $Q(\tau) \in \F(C)$ for any $\tau \in W$, and if $n$ is odd, then $Q(\tau) \notin \F(C)$ for any $\tau \in W \setminus \{ \tau_0 \}$.
Consider some $\tau \in (0,1)$ sufficiently close to $\tau_0$. Then, as $I$ is open, we have  $\tau \in I$. Imagine a billiard ball at $q(\tau)$ and push it parallel to $q(\tau_0) q_1(\tau_0)$. The ball bounces back from the side $p_2p_3$ at $q_1(\tau)$ by the optic reflection law, and runs parallel to $q_1(\tau_0)q_2(\tau_0)$. Here an elementary computation shows that, also by the optic reflection law, the signed distance of the parallel segments $q(\tau)q_1(\tau)$ and $q(\tau_0)q_1(\tau_0)$ is the opposite of the signed distance between $q_1(\tau)q_2(\tau)$ and $q_1(\tau_0)q_2(\tau_0)$ (cf. Figure~\ref{fig:Fagnano} for an illustration of this phenomenon with a regular hexagon as $C$). Repeating this consideration for all sides of $Q(\tau_1)$, we obtain that the billiard trajectory ends at $q(\tau)$ if and only if the ball bounces back an odd number of times; that is, if $n$ is even.
\end{proof}

\begin{rem}\label{rem:evenn}
The proof of (ii) of Theorem~\ref{thm:perim_struct} yields a little more: if $n$ is even and there are infinitely many polygons $Q \in \F(C)$ with $s(Q) = NN \ldots N$, then these polygons are of the form $Q(\tau)$ for some subinterval $I' \subseteq (0,1)$. Indeed, the property follows from the observation that for any edge $p_ip_{i+1}$ the values of $\tau$ such that after reflections the corresponding point lies on $p_ip_{i+1}$ is an interval, and the intersection of intervals is an interval.
\end{rem}

\begin{rem}\label{rem:alln}
Assume that there is some $Q \in \F(C)$ with $s(Q) = NN \ldots N$. Then the boundary of $Q$ coincides with some $\Gamma(\tau_0)$ satisfying the properties in (a) and (b).
In the other direction, if some $\Gamma(\tau_0)$ satisfies the properties in (a) and (b), then, applying the reflection argument as in the second part of the proof of Theorem~\ref{thm:perim_struct} and using the elementary fact that the length of any polygonal path connecting two points is at least as large as the distance between the points, it follows that $\Gamma(\tau_0)$ is the boundary of some $Q \in \F(C)$ with $s(Q) = NN \ldots N$. In particular, this implies that all elements $Q \in \F(C)$ with $s(Q) = NN \ldots N$, and also their perimeter, can be found in $O(n)$ steps. Indeed, carrying out the reflections to $p_np_1$, we can compute the function $\tau \mapsto |q(\tau)q^*(\tau)|$ and determine its unique local minimum in $O(n)$ steps.
\end{rem}

\begin{rem}\label{rem:regular}
Let $C$ be a regular $n$-gon, and let $Q$ be the convex hull of the midpoints of the edges of $C$. From Theorem~\ref{thm:perim_struct} and Remark~\ref{rem:alln} it follows that $Q$ is a minimum area convex polygon inscribed in $C$. Indeed, with the notation of the proof of Theorem~\ref{thm:perim_struct} $Q$ coincides with the boundary of $\Gamma ( 1/2 )$. On the other hand, by the symmetry of $C$, we have that $\Gamma(\tau)$ and $\Gamma(1-\tau)$ are congruent, implying that the function $\tau \mapsto |q(\tau)q^*(\tau)| = |q^*(\tau)-q(\tau)|$ is symmetric to $1/2$. This yields that either $q^*(\tau)-q(\tau)$ is independent of $\tau$, or it moves on a line perpendicular to $q^*(1/2)-q(1/2)$. In both cases, we obtain that the length of $\Gamma(1/2)$ is the minimum of $|q(\tau)q^*(\tau)|$ for all $\tau \in \Re$. By Remark~\ref{rem:alln}, this implies that $Q \in \F(C)$.
\end{rem}

In the remaining part of Subsection~\ref{subsec:perim_algorithm}, we present an algorithm to find an element of $\F(C)$ and its perimeter.
Our algorithm is based on the one in \cite{ADLT19}, with the necessary modifications.

For any $i,j$ with $i< j \leq i+n$, let $\Gamma_{ij}$ be a shortest polygonal curve $\bigcup_{t=1}^{m-1} q_tq_{t+1}$ such that
\begin{itemize}
\item[(i)] $q_1=p_i$, $q_m=p_j$, and all vertices of $\Gamma_{ij}$ are boundary points of $C$,
\item[(ii)] all sides $p_ip_{i+1}, p_{i+1}p_{i+2}, \ldots, p_{j-1}p_j$ contain at least one vertex of $\Gamma_{ij}$, and
\item[(iii)] the points $q_1,q_2,\ldots, q_m$ are in this councerclockwise order in the boundary of $C$.
\end{itemize}
Let $\Pi_{ij}$ denote the length of $\Gamma_{ij}$. Clearly, if $Q \in \F(C)$, and $s(Q) \neq NN \ldots N$, then $\perim(Q) = \Pi_{i,i+n}$ for some value of $i$. We present a recursive algorithm which computes $\Pi_{ij}$ for all $i < j \leq i+n$.

First, clearly, we have $Q_{ij}=p_ip_{i+1}$ for $j=i+1$, and for $j=i+2$, $Q_{ij} = p_ip_{i+2}$. This implies that $\Pi_{i,i+1}= |p_ip_{i+1}|$ and  $\Pi_{i,i+2} = |p_ip_{i+2}|$ for all values of $i$.
Consider some $2 < k \leq n$, and assume  that we have computed all values $\Pi_{st}$ with $s < t < s+k$. Choose some $i,j$ with $j=i+k$.
We distinguish between $k$ types of the shortest polygonal curves $\Gamma_{ij}$ satisfying the properties in the list in the previous paragraph.

Type (0): None of the points $p_{i+1},p_{i+2},\ldots, p_{j-1}$ is a vertex of $\Gamma_{ij}$.

Type ($u$): The point $p_{i+u}$ is a vertex of $\Gamma_{ij}$ for some $1 \leq u \leq k-1$.

Here, if $\Gamma_{ij}$ has Type (0), then it has no Type ($u$) for any $1 \leq u \leq k-1$. On the other hand, in general $\Gamma_{ij}$ may have Type ($u$) for more than one distinct value of $u$. Our algorithm determines the value of $\Pi_{ij}$ depending on the type of $\Gamma_{ij}$.

If $\Gamma_{ij}$ has Type (0), then the vertices $q_2, q_3, \ldots, q_{m-1}$ satisfy the optic reflection law by Lemma~\ref{lem:reflection}. Then, using the reflections described in the first part of Theorem~\ref{thm:perim_struct}, the existence of a polygonal curve satisfying these conditions, and in case of existence the vertices of this curve and its length can be found in $O(k)$ steps. Assume that $\Gamma_{ij}$ has Type ($u$) for some $1 \leq u \leq k-1$. Then $\Gamma_{ij}$ is the union of some shortest polygonal curves $\Gamma_{i,i+u}$ and $\Gamma_{i+u,j}$, and its length is $\Pi_{i,i+u}+\Pi_{i+u,j}$. To find the shortest polygonal curves having Type ($u$) for all possible values of $u$, and the length of these curves, we need $O(k)$ steps.

Starting with $k=3$, for every fixed value of $k$ we execute the above procedure for all $1\leq i \leq n$. Then we increase the value of $k$ by one and repeat all steps until $k=n$. Thus, we obtain the values of $\Pi_{ij}$ for all $i,j$ with $i<j\leq i+n$ in $O(n^3)$ steps. Indeed, we have seen that for any fixed values $i$ and $j=i+k$, we can find the value of $\Pi_{ij}$ in $O(k)$ times, and hence, the estimate $O(n^3)$ follows by executing this procedure for all values of $i$ and $k$, and using the inequality $k \leq n$.

Let $p_1(C) = \min \{ \Pi_{i,i+n} : i=1,2,\ldots, n \}$. We need to handle the case of the convex polygons $Q \in \F(C)$ with $s(Q) = NN \ldots N$.
Nevertheless, by Remark~\ref{rem:alln}, all such polygons, if they exist, and their perimeter can be found in $O(n)$ steps.
In conclusion, $p(C)$, and also an inscribed convex polygon with minimum perimeter, can be found in $O(n^3)$ steps.

We note that by (ii) of Theorem~\ref{thm:perim_struct}, if $n$ is even and there is some $Q \in \F(C)$ with $s(Q) = NN \ldots N$ then 
there is some $Q \in \F(C)$ with $s(Q) \neq NN \ldots N$. This yields that if $n$ is even then $p(C) = p_1(C)$; thus, if we want to calculate only the value of $p(C)$ we can skip the last step of the algorithm for $n$ even. 

\begin{rem}\label{rem:perim_dectree}
As in case of minimum area polygons (cf. Remark~\ref{rem:area_dectree}), our algorithm can be carried out in such a way that we keep track of all minimum perimeter inscribed polygons. These polygons are not listed by the algorithm, but represented in the form of a decision tree.
\end{rem}

\subsection{Combinatorial properties}\label{subsec:perim_comb}

In Subsection~\ref{subsec:perim_algorithm} we have seen that, apart from the sequence containing only $N$s, for any cyclic sequence $s \in \{ U,N \}^n$ there is at most one minimum perimeter polygon $Q \in \F(C)$ whose associated cyclic sequence is $s$. Our main goal in this subsection is to determine the cyclic sequences $s \in \{ N,U \}^n$ with the property that for a suitable convex $n$-gon $C$ there is some $Q \in \F(C)$ such that $s(Q) = s$.

Let $s \in \{U,N\}^n$, and let us call it realizable if it is associated to some convex polygon $Q \in \F(C)$ with a suitable choice of $C$. Similarly like in Subsection~\ref{subsec:area_comb}, if $s$ is realizable, then it contains no three consecutive $U$s.
Indeed, if three consecutive vertices of $C$ are used, then the perimeter of $Q$ could be further reduced by removing the middle vertex from the vertex set of $Q$.
Our main result is the following.

\begin{thm}\label{thm:perim_comb}
A cyclic sequence $s \in \{ U, N \}^n$ is realizable if and only if $s$ does not contain three consecutive $U$s.
\end{thm}

\begin{proof}
Since no realizable sequence contains three consecutive $U$s, we need to prove that if $s \in \{ U,N \}^n$ does not contain three consecutive $U$s, then $s$ is realizable. Since the cyclic sequence $s = NN \ldots, N$ is clearly realized by a regular polygon, we assume that $s \neq  NN \ldots N$.

Let the number of $N$s in such a sequence $s \in \{ U,N \}^n$ be $k$. First, we prove the assertion in the case that $k \geq 3$. Let $C_0$ be a regular $k$-gon of unit edge length, with vertices $p_1, p_2, \ldots, p_k$, where the indices are understood mod $k$, and let $Q_0$ be the convex hull of the midpoints of the sides of $C_0$. Let the vertices of $Q_0$ be $m_1, m_2, \ldots, m_k$ such that $m_i \in p_ip_{i+1}$. By Remark~\ref{rem:regular}, if $k$ is odd, then $Q_0$ is the unique smallest perimeter convex polygon inscribed in $C_0$. Furthermore, it follows from Remark~\ref{rem:evenn} that if $k$ is even, then for any $\tau \in [0,1]$, there is a unique convex polygon $Q(\tau)$, with vertices $q_1(\tau), q_2(\tau), \ldots, q_n(\tau)$ in counterclockwise order, such that $q_1(\tau)= \tau p_1 + (1-\tau) p_2$ and $q_i(\tau)q_{i+1}(\tau)$ is parallel to $m_im_{i+1}$ for all values of $i$. Furthermore, the perimeters of these polygons are equal, $\F(C_0) = \{ Q(\tau) : \tau \in [0,1] \}$, where $Q(1/2)=Q_0$, and in the degenerate cases $\tau=0$ and $\tau =1$, we have $q_1(0)=q_2(0)$, and $q_1(1)=q_n(1)$.

We note that since the perimeter of a polygon is a continuous function of its vertices, any convex polygon inscribed in $C_0$ whose perimeter is `close to' $p(C_0)$
is `close to' an element of $\F(C_0)$. Moreover, in the family of convex $k$-gons for any fixed value of $k$, $p(C)$ is a continuous function of $C$, implying that if $C_1$ and $C_2$ are convex $k$-gons and $C_1$ is `close to' $C_2$, then $p(C_1)$ is `close to' $p(C_2)$.

Let $\zeta > 0$ be some sufficiently small fixed value. We define an auxiliary, degenerate convex $n$-gon $C'$ in the following way. A side $S_i = p_ip_{i+1}$, $i=1,2,\ldots,k$ of $C_0$ has Type (t) for some $t \in \{ 0,1,2 \}$, if the $i^{\mathrm{th}}$ and the $(i+1)^{\mathrm{st}}$ $N$s in $s$ are separated by $t$ $U$s.
If $S_i$ has Type (0), we regard $S_i$ as a side of $C'$. If $S_i$ has Type (1), we regard $m_i$ as a vertex, and the segments $S_i'=p_i m_i$, $S_i''=m_i p_{i+1}$ as sides of $C'$. Assume that $S_i$ has Type (2). Then we choose two points $p_i'$ and $p_i''$ on $S_i$ symmetric to $m_i$ such that $p_i, p_i', p_i'', p_{i+1}$ are in this linear order on $S_i$, and $|p_i'p_i''|=\zeta$. We regard $p_i',p_i''$ as vertices and $p_ip_i', p_i'p_i'', p_i''p_{i+1}$ as sides of $C'$. We call a (possibly degenerate) convex polygon $Q'$ a polygon \emph{inscribed} in $C'$ if every side of $C'$ contains at least one vertex of $Q'$. By continuity it follows that if $\zeta$ is sufficiently small, then no minimum perimeter polygon inscribed in $C'$ contains a vertex of $C_0$. On the other hand, applying the idea of the proof of Lemma~\ref{lem:reflection}, we obtain that any such polygon contains all the $p_i'$s and $p_i''$s on the Type (2) sides as well as the midpoints of the Type (1) sides of $C_0$ as vertices. Thus, there is a unique minimum perimeter polygon inscribed in $C'$ and the vertex set of this polygon consists of all $p_i'$s and $p_i''$s on the Type (2) sides, the midpoints of the Type (1) sides of $C_0$, and one point in the interior of each Type (0) side of $C_0$.

Now let $u_i$ be the outer unit normal vector of $S_i$ for all values of $i$. Consider some $\delta > 0$. If $S_i$ is of Type (1), set $\bar{p}_i'=m_i+ \varepsilon u_i$, and if $S_i$ is of Type (2), set $\bar{p}_i'=p_i'+ \varepsilon u_i$ and $\bar{p}_i''=p_i''+ \varepsilon u_i$. Let $C$ denote the convex hull of the union of $C_0$ and all points $\bar{p}_i'$ and $\bar{p}_i''$. Then, by continuity and Lemma~\ref{lem:reflection}, if $\delta > 0$ is sufficiently small, then $C$ is a convex polygon, and any minimum perimeter convex polygon $Q$ inscribed in $C$ has all $\bar{p}_i'$s and $\bar{p}_i''$s as vertices, and no vertex of $C_0$ is a vertex of $C$.
Clearly, the sequence assigned to any such polygon is $s$, and thus, Theorem~\ref{thm:perim_struct} yields that such a $Q$ is unique.

\begin{figure}[ht]
\begin{center}
\includegraphics[width=0.5\textwidth]{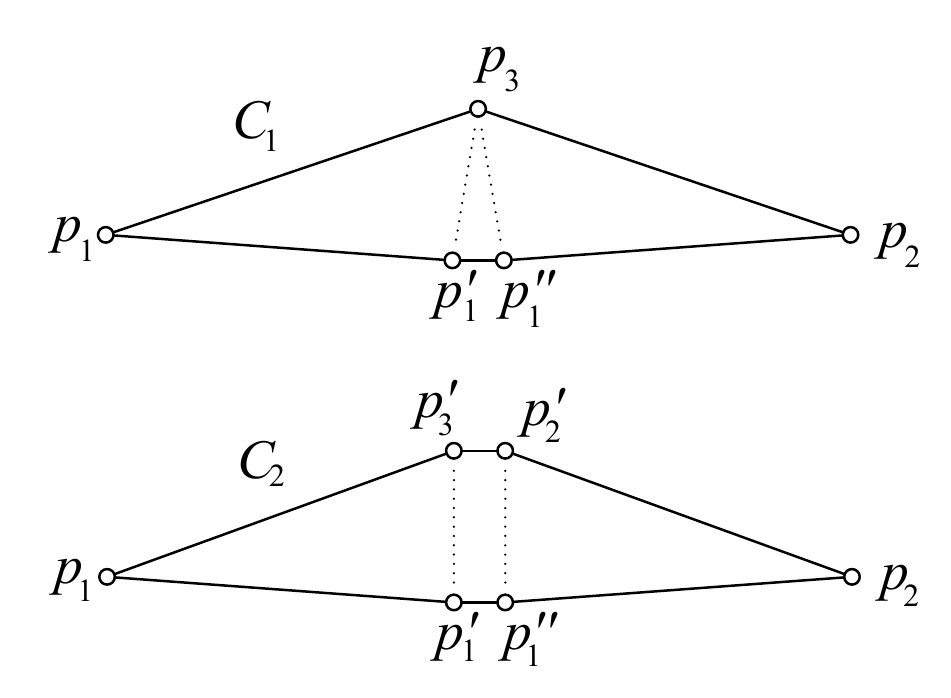}
\caption{The polygons $C_1$ and $C_2$ constructed in the proof of Theorem~\ref{thm:perim_comb}. The parts of the boundaries of the inscribed polygons $Q_1$ and $Q_2$ in the interiors of the polygons are denoted by dotted lines.}
\label{fig:perim_comb}
\end{center}
\end{figure}

We are left with the case that the number $k$ of $N$s in $s$ is at most $2$. By our conditions, from this we have that $s=NUUNU$ or $s=NUUNUU$. In these cases, we use the fact that for an obtuse isosceles triangle $C_0$ with base $p_1p_2$ and apex $p_3$, the shortest closed polygonal curve containing a point from each side of $C_0$ is a degenerate `double segment' containing $p_3$ and the midpoint of $p_1p_2$. By slightly modifying $C_0$ as in the previous consideration, we obtain a convex pentagon $C_1$ and a convex hexagon $C_2$ such that $\F(C_i)$ consists of a unique element $Q_i$ for $i=1,2$, and $s(Q_1)=NUUNU$ and $s(Q_2)=NUUNUU$ (see Figure~\ref{fig:perim_comb}).
\end{proof}

\section{Additional results about circumscribed polygons}\label{sec:others}

Similarly like in \cite{ADLT19}, we may consider the problem of maximizing a geometric quantity among convex polygons circumscribed about a given convex polygon. For this purpose, let us recall Definition 1 from \cite{ADLT19}.

\begin{defn}\label{defn:circumscribed}
Let $C \subset \Re^2$ be a convex $n$-gon. If $Q$ is a convex $m$-gon that
contains $C$ and each vertex of $C$ is on the boundary of $Q$,
then we say that $Q$ is
\emph{circumscribed} about $C$.
\end{defn}

As in \cite{ADLT19}, we intend to exclude the existence of unbounded convex polygonal regions $Q$, containing $C$, with the property that each vertex of $C$ lies on a side of $Q$. Thus, we assume that the sum of any two consecutive angles of $C$ is greater than $\pi$. In particular, from this it follows that $n \geq 5$.
In this section we collect our observations about convex polygons circumscribed about $C$ with maximal perimeter or maximal diameter. We must add that our observations are not sufficient in either problem to provide a complete algorithm to find such a polygon.
We also remark that the problem of finding maximum area circumscribed polygons was the main topic of the paper \cite{ADLT19} motivating our research. Furthermore, the problem of finding a minimum diameter polygon inscribed in a given convex polygon seems different from the `dual' problem discussed in Subsection~\ref{subsec:maxdiam}.

Finally, we note that while a convex polygon $Q$ circumscribed about a convex $n$-gon may have arbitrarily many sides, since by removing the sides of $Q$ not containing a vertex of $P$ we increase the perimeter of $Q$ and do not decrease its diameter, it follows that if $Q$ has maximal perimeter among circumscribed polygons, it has at most $n$ vertices, and the same property holds for at least one circumscribed polygon with maximal diameter.

\subsection{Maximum perimeter convex polygons circumscribed about $C$}

Set $P(C) = \sup \{ \perim(Q) : Q \hbox{ is circumscribed about } C \}$, and note that by our conditions, $P(C)$ exists, and it is attained by a convex polygon $Q$.
Our main result is as follows.

\begin{thm}\label{thm:circum_perim}
Let $Q$ be a maximum perimeter convex polygon circumscribed about $C$. Let the vertices of $Q$ in counterclockwise order be $q_1, q_2, \ldots, q_m$, with the vertices understood mod $m$. For all values of $i$, let the measure of the angle of $Q$ at $q_i$ be denoted by $\beta_i$.
Then every side $q_iq_{i+1}$ of $Q$ contains at least one and at most two vertices of $C$, and if $p_j$ is the unique vertex of $C$ on $q_iq_{i+1}$, then
\begin{equation}\label{eq:circum_perim}
|q_i p_j|  \cot\left( \beta_i \right) = |q_{i+1}p_j| \cot\left( \beta_{i+1}\right).
\end{equation}
\end{thm}

\begin{proof}
Note that if $q_iq_{i+1}$ contains no vertex of $C$, then the remaining sidelines of $Q$ are the sidelines of a convex polygon $Q'$ circumscribed about $C$ and satisfying the inequality $\perim(Q) < \perim(Q')$. This shows that $q_iq_{i+1}$ contains at least one vertex of $C$, and a similar argument excludes the case that $q_iq_{i+1}$ contains a unique vertex of $C$ coinciding with $q_i$ or $q_{i+1}$. The fact that any such side contains at most two vertices of $C$ follows from the convexity of $C$. Thus, we may assume that $q_iq_{i+1}$ contains exactly one vertex $p_j$ of $C$, which is in the interior of $q_iq_{i+1}$.

We show that in this case (\ref{eq:circum_perim}) holds.
Let $L_{i-1}$, $L_i$ and $L_{i+1}$ denote the lines through $q_{i-1}q_i$, $q_iq_{i+1}$ and $q_{i+1}q_{i+2}$, respectively.
Let $d_{i-1}$ and $d_{i+1}$ denote the distances of $p_j$ from $L_{i-1}$ and $L_{j+1}$, respectively.

\begin{figure}[ht]
\begin{center}
\includegraphics[width=0.85\textwidth]{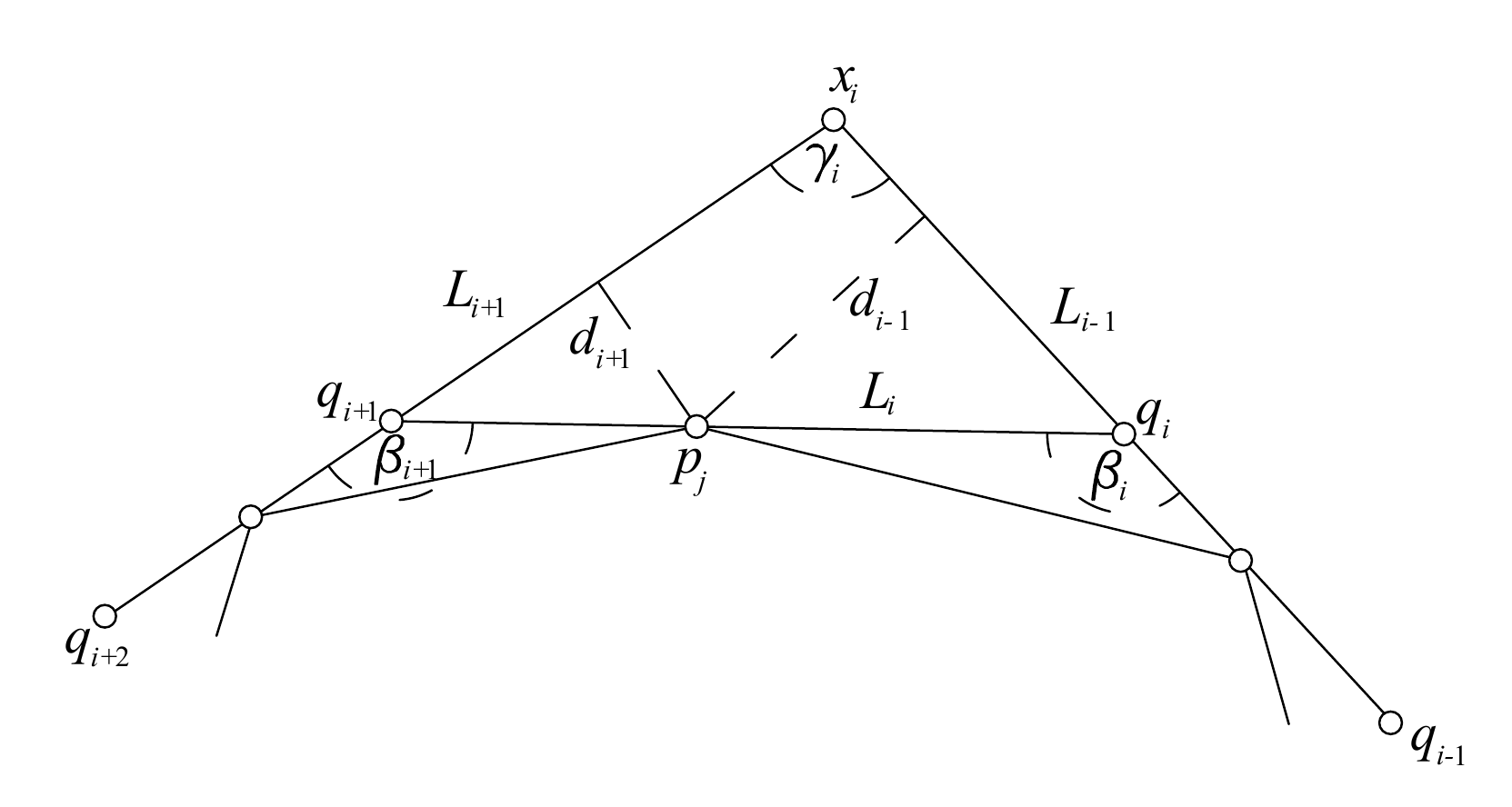}
\caption{Notations for the proof of Theorem~\ref{thm:circum_perim}.}
\label{fig:circumscribed}
\end{center}
\end{figure}

Consider the case that $L_{i-1}$ and $L_{i+1}$ intersect, and their intersection point $x_i$ is separated from $C$ by $L_i$. Let the measure of the angle $\angle q_i x_i q_{i+1}$ be denoted by $\gamma_i$ (cf. Figure~\ref{fig:circumscribed}). 
An elementary computation shows that $|q_ip_j| + |p_jq_{i+1}| = \frac{d_{i-1}}{\sin \beta_i} + \frac{d_{i+1}}{\sin \beta_{i+1}}$. Now, let us change $Q$ by rotating $L_i$ around $p_j$. Then the above expression can be regarded as a function $f(\beta_i)$, with $\beta_{i+1}= \pi + \gamma_i - \beta_i$, and with the values of $d_{i-1}$, $d_{i+1}$ and $\gamma_i$ fixed. By differentiating $f$, we obtain that it is a strictly concave function of $\beta_i$, and its unique maximum is attained if
$0=\frac{d_{i-1} \cos \beta_i}{\sin^2 \beta_i} - \frac{d_{i+1} \cos \beta_{i+1}}{\sin^2 \beta_{i+1}} = |q_i p_j|  \cot\left( \beta_i \right) - |q_{i+1}p_j| \cot\left( \beta_{i+1}\right)$.

If $L_{i-1}$ and $L_{i+1}$ do not intersect, or their intersection point is not separated from $C$ by $L_i$, a similar argument can be applied.
\end{proof}

In Remark~\ref{rem:circum_perim} we use the notation in Theorem~\ref{thm:circum_perim}.

\begin{rem}\label{rem:circum_perim}
Consider some point $q_i$ such that the line $L$ through $q_ip_j$ supports $C$. Then it can be shown that $L$ contains at most one point $q_{i+1}$ satisfying the condition in (\ref{eq:circum_perim}), and the coordinates of this point can be computed from the coordinates of $q_i$ and the vertices of $C$ in $O(1)$ steps. Nevertheless, since the condition determining this point seems more complicated than in Sections~\ref{sec:area} and \ref{sec:perim}, to characterize the convex polygons circumscribed about $C$ and satisfying the conditions in Theorem~\ref{thm:circum_perim} seems to be a more difficult problem than in the cases investigated in this paper.
\end{rem}

\subsection{Maximum diameter convex polygons circumscribed about $C$}\label{subsec:maxdiam}

Compared to the case of perimeter, it seems much easier to find the maximum diameter of the convex polygons circumscribed about $C$. Indeed, let $x_i$ denote the intersection point of the lines through $p_{i-1}p_i$ and $p_{i+1}p_{i+2}$, and note that $q_i$ exists by our conditions for $C$. Then at least one circumscribed polygon of maximal diameter can be obtained
as the convex hull of some points $q_i \in \conv \{ p_i, p_{i+1}, x_i \}$, $i=1,2,\ldots,n$. Thus, the diameter of any convex polygon circumscribed about $C$ is less than or equal to the diameter of $X=\conv \{ x_i : i=1,2,\ldots, n \}$, which we denote by $\diam (X)$. Here the endpoints of any diameter of $X$ are vertices of $X$. On the other hand, if $X$ has a diameter $xy$ whose endpoints $x,y$ are not consecutive vertices of $X$, then there is a convex polygon $Q$ circumscribed about $C$ whose diameter is $\diam (X)$.
Hence, to find the maximum diameter of the convex polygons circumscribed about $C$, in `many' cases it is sufficient to find $\diam (X)$.

We note that $\diam (X)$ can be computed in $O(n)$ steps. To do it, first we compute the points $x_i$ and then the vertices of $X$ by Graham's convex hull algorithm \cite{Graham}. If the points $x_i$ are already ordered according to angles in polar coordinates, as in our case, the running time of this algorithm is $O(n)$.
As the last step, the diameters of $X$ can be computed from the vertices of $X$ by the rotating calipers algorithm of Shamos \cite{Shamos}.
Nevertheless, to give a complete algorithm to find the maximum diameter of the convex polygons circumscribed about $C$, the remaining case, when the diameters of $X$ connect consecutive $x_i$s, must also be handled.

\noindent
\textbf{Acknowledgements.}\\
The authors express their gratitude to Bal\'azs Keszegh for many useful comments, and the anonymous referees for many helpful suggestions.

\noindent
\textbf{Funding}: The second author was supported by the National Research, Development and Innovation Office, NKFI, K-119670, the J\'anos Bolyai Research Scholarship of the Hungarian Academy of Sciences, and the BME IE-VIZ TKP2020 and \'UNKP-20-5 New National Excellence Programs by the Ministry of Innovation and Technology.


\end{document}